\documentclass[reqno]{amsart}

\usepackage{a4wide}
\usepackage{color}
\usepackage{mathrsfs}
\usepackage{mathtools}
\usepackage{tikz-cd}
\usepackage{amsmath}
\usepackage{amssymb}
\usepackage{bbm}
\usepackage{esint}
\numberwithin{equation}{section}
\usepackage[colorlinks,citecolor=green,linkcolor=red]{hyperref}
\usepackage{hyphenat}
\usepackage{soul}

\usepackage[latin1]{inputenc}

\newcommand{\N}{\mathbb{N}}
\newcommand{\R}{\mathbb{R}}
\newcommand{\sfd}{{\sf d}}
\renewcommand{\d}{{\mathrm d}}

\newcommand{\restr}[1]{\lower3pt\hbox{\(|_{#1}\)}}

\newcommand{\nchi}{{\raise.3ex\hbox{\(\chi\)}}}

\newcommand{\fr}{\penalty-20\null\hfill\(\blacksquare\)}
\newcommand{\X}{{\rm X}}
\newcommand{\Y}{{\rm Y}}

\newcommand{\mm}{\mathfrak m}

\newcommand{\ppi}{\boldsymbol{\pi}}

\newtheorem{theorem}{Theorem}[section]
\newtheorem{corollary}[theorem]{Corollary}
\newtheorem{lemma}[theorem]{Lemma}
\newtheorem{proposition}[theorem]{Proposition}
\newtheorem{definition}[theorem]{Definition}

\newtheorem{remark}[theorem]{Remark}

\title[Convergence of metric measure spaces and the Urysohn universal space]{Convergence of metric measure spaces via \\ embeddings in the Urysohn universal space}

\author{Milica Cakovi\'{c}}
\address{Department of Mathematics and Statistics,
P.O.\ Box 35, FI-40014 University of Jyvaskyla; Department of Mathematics and Informatics, Faculty of Sciences, University of Novi Sad, Trg Dositeja Obradovi\'ca 3, 21000 Novi Sad, Serbia}
\email{milica.m.cakovic@jyu.fi}

\author{Enrico Pasqualetto}
\address{Department of Mathematics and Statistics,
P.O.\ Box 35, FI-40014 University of Jyvaskyla}
\email{enrico.e.pasqualetto@jyu.fi}

\author{Timo Schultz}
\address{Department of Mathematics and Statistics,
P.O.\ Box 35, FI-40014 University of Jyvaskyla}
\email{timo.m.schultz@jyu.fi}

\begin{document}
\date{\today}
\begin{abstract}
We study different notions of convergence of metric measure spaces by means of isometric embeddings into the Urysohn
universal metric space \(\mathbb U\). Due to the universality of \(\mathbb U\), the collection \(\mathbb X_1\) of isomorphism
classes of normalised metric measure spaces can be canonically identified with the quotient (set)
\(\mathscr P_\sim(\mathbb U)=\mathscr P(\mathbb U)/\sim\) of the space \(\mathscr P(\mathbb U)\) of Borel probability
measures on \(\mathbb U\), where \(\mu\sim\nu\) if \(\nu\) is the pushforward of \(\mu\) under an isometry between their
respective supports. By making crucial use of the ultrahomogeneity of \(\mathbb U\), we show that, under the above identification,
Gromov's box topology on \(\mathbb X_1\) coincides with the quotient topology induced by the weak topology of \(\mathscr P(\mathbb U)\).
More quantitatively, the truncated \(1\)-Wasserstein distance on \(\mathscr P(\mathbb U)\) induces a complete and separable
distance \(\sfd_{\rm mG}\) on \(\mathbb X_1\cong\mathscr P_\sim(\mathbb U)\), which metrises the quotient topology of
\(\mathscr P_\sim(\mathbb U)\) and is H\"{o}lder equivalent to the box distance \(\square\).
\end{abstract}
\maketitle
\section{Introduction}
\subsection{General overview}
A variety of notions of convergence for metric measure spaces have been introduced and studied in the literature,
beginning with Gromov's pioneering work \cite{Gromov07}, where -- among other notions -- the \emph{box topology}
was introduced. One of the main motivations for introducing such notions came from the study of limits of smooth
manifolds, which may develop singularities, collapse, or exhibit concentration phenomena. Gromov's theory of
convergence and concentration of metric measure spaces was developed further by Shioya \cite{Shioya16}.
Another important notion of convergence is the \emph{measured-Gromov--Hausdorff convergence} introduced by Fukaya \cite{Fuk87}.
Two further notions that will be considered in this paper are Sturm's \emph{\({\bf D}\)-distance} \cite{Sturm2006a}
and the \emph{pointed-measured-Gromov convergence} introduced by Gigli, Mondino and Savar\'{e} in \cite{GigliMondinoSavare15}.
\medskip

Most notions of convergence of metric measure spaces admit an equivalent
extrinsic characterisation, in which convergence is described by simultaneously embedding the spaces into a common ambient metric space.
The aim of this paper is to investigate the implications of choosing
the \emph{Urysohn universal metric space} \((\mathbb U,\sfd_{\mathbb U})\) as such an ambient space. Introduced by Urysohn
in a paper published posthumously \cite{Urysohn27}, the space \(\mathbb U\) is the unique -- up to isometry
-- complete and separable metric space that is \emph{universal} (i.e., it contains an isometric copy of every
separable metric space) and \emph{ultrahomogeneous} (i.e., every isometry between finite subsets of \(\mathbb U\)
extends to a global isometry of \(\mathbb U\)); see Section \ref{s:Urysohn_univ_space} for further details.
\subsection{Statement of results}
Letting \(\mathbb X_1\) be the space of all isomoprhism classes of \emph{normalised} metric measure spaces (see Section \ref{s:mms}),
we consider a notion of convergence for sequences in \(\mathbb X_1\) that
we refer to as the \emph{measured-Gromov convergence} (or \emph{mG-convergence}). Akin to the extrinsic definition
of pmG-convergence, we say that a sequence \(([\X_n,\sfd_n,\mm_n])_n\subseteq\mathbb X_1\)
mG-converges to \([\X_\infty,\sfd_\infty,\mm_\infty]\in\mathbb X_1\) if \((\iota_n)_\#\mm_n\) weakly converges
to \((\iota_\infty)_\#\mm_\infty\) for some choice of isometric embeddings
\(\iota_n\colon\X_n\to\Y\) into some ambient metric space \(\Y\); see Definition \ref{def:mG}.
Due to the universality of \(\mathbb U\), the space \(\mathbb X_1\) can be identified with the quotient
\(\mathscr P_\sim(\mathbb U)\coloneqq\mathscr P(\mathbb U)/\sim\), where the equivalence relation \(\sim\)
on the set of probability measures \(\mathscr P(\mathbb U)\) is defined as follows:
\begin{equation}\label{eq:def_equiv_rel_P(U)}
\mu\sim\nu\qquad\Longleftrightarrow\qquad\phi_\#\mu=\nu\;\text{ for some isometric bijection }
\phi\colon{\rm spt}(\mu)\to{\rm spt}(\nu).
\end{equation}
We denote by \(\Phi\colon\mathbb X_1\to\mathscr P_\sim(\mathbb U)\) the canonical bijection
(see Definition \ref{def:Phi}). Our main results are:
\begin{itemize}
\item[1)] Letting \(W\) denote the \emph{truncated \(1\)-Wasserstein distance} on \(\mathscr P(\mathbb U)\)
-- see \eqref{eq:trunc_1-Wasserstein} for its definition -- we consider the Hausdorff metric \(W_H\) induced
by \(W\) on the quotient \(\mathscr P_\sim(\mathbb U)\):
\begin{equation}\label{eq:def_W_H}
W_H(\theta,\eta)\coloneqq\max\bigg\{\sup_{\mu\in\theta}\inf_{\nu\in\eta}W(\mu,\nu),\sup_{\nu\in\eta}\inf_{\mu\in\theta}W(\mu,\nu)\bigg\}
\quad\text{ for every }\theta,\eta\in\mathscr P_\sim(\mathbb U).
\end{equation}
We prove that \((\mathscr P_\sim(\mathbb U),W_H)\) is a complete separable metric space and that \(W_H\) generates
the quotient topology (induced by the weak topology of \(\mathscr P(\mathbb U)\), which is metrised by \(W\)).
\item[2)] Via the bijection \(\Phi\), the distance \(W_H\) induces a distance \(\sfd_{\rm mG}\) on \(\mathbb X_1\)
(see \eqref{eq:def_dmG}). We prove that \(\sfd_{\rm mG}\) induces precisely the notion of
mG-convergence of sequences in \(\mathbb X_1\).
\item[3)] Denoting by \(\square\) the box distance on \(\mathbb X_1\) (see Section \ref{s:box_dist}),
we prove that \(\sfd_{\rm mG}\) and \(\square\) are H\"{o}lder equivalent. In particular, the space \((\mathbb X_1,\square)\)
is homeomorphic to the topological quotient \(\mathscr P_\sim(\mathbb U)\) of \((\mathscr P(\mathbb U),W)\).
In other words, mG-convergence provides an extrinsic formulation of convergence with respect to the box topology.
\end{itemize}
The results stated in 1), 2) and 3) will be established in Theorem \ref{thm:equiv_W_H}, Theorem
\ref{thm:equiv_mG-conv} and Theorem \ref{thm:box_equiv_mG}, respectively. In the course of proving these results,
we also provide further characterisations of the distance \(\sfd_{\rm mG}\). For example, in Proposition \ref{prop:D=d_mG}
we show that \(\sfd_{\rm mG}\) coincides with a truncated version of the \emph{\(L_1\)-transportation distance} introduced by Sturm
in \cite{Sturm2006a}. Moreover, taking advantage of the special features of \(\mathbb U\), in Theorem
\ref{thm:precpt_W_H} we provide a characterisation of all precompact subsets of \((\mathscr P_\sim(\mathbb U),W_H)\)
with a relatively simple proof. As a consequence, we immediately recover Gromov's \emph{compactness criterion}
\cite[Lemma 4.28]{Shioya16} for the box topology (see Corollary \ref{cor:compact_box_top}). Finally,
embedding into the Urysohn universal space has interesting consequences also for other notions of convergence.
For example, in Theorem \ref{thm:embed_for_pmG} we prove -- by means of the main results
of this paper -- that the space \(\mathbb X_{pt}\) of equivalence classes of \emph{pointed} metric measure spaces
endowed with the pmG-topology can be realised as a subspace of the topological product
\(\mathbb X_1^{\mathbb Z}\times\R^{\mathbb Z}\), where \(\mathbb X_1\) is equipped with the box topology.
\subsection*{Acknowledgements}
M.\ Cakovi\'{c} was supported by the Research Council of Finland grant 355122.
E.\ Pasqualetto was supported by the Research Council of Finland grant 362898.
\section{Preliminaries}
\subsection{Quotients of a metric space}\label{s:quotients_metric_space}
Let \((\X,\sfd)\) be a metric space and \(\mathcal R\) an equivalence relation on \(\X\). Given any two
elements \(a\) and \(b\) of the quotient space \(\X/\mathcal R\), we define
\[\begin{split}
\sfd_H(a,b)&\coloneqq\max\bigg\{\sup_{x\in a}\inf_{y\in b}\sfd(x,y),\sup_{y\in b}\inf_{x\in a}\sfd(x,y)\bigg\},\\
{\rm dist}_\sfd(a,b)&\coloneqq\inf_{(x,y)\in a\times b}\sfd(x,y),\\
\bar\sfd(a,b)&\coloneqq\inf_{(x_i)_i,(y_i)_i}\sum_{i=1}^n\sfd(x_i,y_i),
\end{split}\]
where the infimum in the definition of \(\bar\sfd(a,b)\) is taken over all \(n\in\N\) and \((x_i)_{i=1}^n,(y_i)_{i=1}^n\subseteq\X\)
such that \((x_1,y_n)\in a\times b\) and \(y_i\mathcal R x_{i+1}\) for every \(i=1,\ldots,n-1\). Whereas the \emph{Hausdorff metric}
\(\sfd_H\) is an extended pseudodistance and the \emph{quotient metric} \(\bar\sfd\) is a pseudodistance, we have that \({\rm dist}_\sfd\)
is not -- in general -- a pseudodistance (as easy examples illustrate). We observe that
\begin{equation}\label{eq:ineq_quot_pseudodist}
\bar\sfd(a,b)\leq{\rm dist}_\sfd(a,b)\leq\sfd_H(a,b)\quad\text{ for every }a,b\in\X/\mathcal R.
\end{equation}
To prove the first inequality, consider only those competitors for \(\bar\sfd(a,b)\) with \(n=1\).
To prove the second inequality, note that \({\rm dist}_\sfd(a,b)=\inf_{x\in a}\inf_{y\in b}\sfd(x,y)
\leq\sup_{x\in a}\inf_{y\in b}\sfd(x,y)\leq\sfd_H(a,b)\).
\begin{lemma}\label{lem:various_top_on_quot}
Let \((\X,\sfd)\) be a metric space. Then the quotient topology \(\tau_{\X/\mathcal R}\) on \(\X/\mathcal R\) satisfies
\[
\tau[\bar\sfd]\subseteq\tau_{\X/\mathcal R}\subseteq\tau[\sfd_H],
\]
where \(\tau[\bar\sfd]\) and \(\tau[\sfd_H]\) denote the topologies induced by \(\bar\sfd\) and \(\sfd_H\), respectively.
\end{lemma}
\begin{proof}
If \((x_n)_n\subseteq\X\) and \(x\in\X\) satisfy \(\sfd(x_n,x)\to 0\), then \(\bar\sfd([x_n]_{\mathcal R},[x]_{\mathcal R})\leq\sfd(x_n,x)\to 0\)
by \eqref{eq:ineq_quot_pseudodist}. Hence, the quotient map \((\X,\tau[\sfd])\ni x\mapsto[x]_{\mathcal R}\in(\X/\mathcal R,\tau[\bar\sfd])\)
is continuous and thus \(\tau[\bar\sfd]\subseteq\tau_{\X/\mathcal R}\).

To prove that \(\tau_{\X/\mathcal R}\subseteq\tau[\sfd_H]\), it suffices to show that any given \(\tau_{\X/\mathcal R}\)-closed
set \(C\subseteq\X/\mathcal R\) is \(\tau[\sfd_H]\)-closed. To this aim, fix a sequence \((a_n)_n\subseteq C\) and assume that
\(\sfd_H(a_n,a)\to 0\) for some \(a\in\X/\mathcal R\). Take any \(x\in a\) and pick \(x_n\in a_n\) for every \(n\in\N\) so that
\(\sfd(x_n,x)\leq\sfd_H(a_n,a)+1/n\), so that \(\sfd(x_n,x)\to 0\). Since \((x_n)_n\) is contained in the \(\sfd\)-closed set
\(\mathcal C\coloneqq\{y\in\X:[y]_{\mathcal R}\in C\}\), we deduce that \(x\in\mathcal C\) and thus \(a=[x]_{\mathcal R}\in C\).
This proves that \(C\) is \(\tau[\sfd_H]\)-closed, as desired.
\end{proof}
\begin{proposition}\label{prop:act_by_isom_conseq}
Let \((\X,\sfd)\) be a metric space and \(\mathcal R\) an equivalence relation on \(\X\). Let \(G\) be a group
acting on \(\X\) by isometries such that the \(\mathcal R\)-equivalence classes coincide with the closures of
the \(G\)-orbits. Then it holds that \(\bar\sfd={\rm dist}_\sfd=\sfd_H\). In particular, \(\sfd_H\) is a distance
on \(\X/\mathcal R\) that induces the quotient topology \(\tau_{\X/\mathcal R}\). If in addition \((\X,\sfd)\)
is complete, then \((\X/\mathcal R,\sfd_H)\) is complete.
\end{proposition}
\begin{proof}
Fix \(a,b\in\X/\mathcal R\), \(x_0\in a\) and \(y_0\in b\). Given any \(x\in a=\overline{G.x_0}\) and
\(\varepsilon>0\), we can find \(g\in G\) such that \(\sfd(x,g.x_0)<\varepsilon\). Hence, we can estimate
\[
\sfd(x,g.y_0)\leq\sfd(x,g.x_0)+\sfd(g.x_0,g.y_0)<\sfd(x_0,y_0)+\varepsilon,
\]
which yields \(\sfd(x,G.y_0)\leq\sfd(x_0,y_0)\). Similarly, \(\sfd(G.x_0,y)\leq\sfd(x_0,y_0)\) for every \(y\in b\).
This shows that \(\sfd_H(a,b)\leq\sfd(x_0,y_0)\), so that \(\sfd_H(a,b)\leq{\rm dist}_\sfd(a,b)\) by the
arbitrariness of \(x_0\in a\) and \(y_0\in b\). Therefore, for any choice of competitors \((x_i)_{i=1}^n\), \((y_i)_{i=1}^n\)
for \(\bar\sfd(a,b)\), we can estimate
\[
\sfd_H(a,b)=\sfd_H([x_1]_{\mathcal R},[y_n]_{\mathcal R})\leq\sum_{i=1}^n\sfd_H([x_i]_{\mathcal R},[y_i]_{\mathcal R})
\leq\sum_{i=1}^n{\rm dist}_\sfd([x_i]_{\mathcal R},[y_i]_{\mathcal R})\leq\sum_{i=1}^n\sfd(x_i,y_i),
\]
whence the inequality \(\sfd_H(a,b)\leq\bar\sfd(a,b)\) follows. Recalling \eqref{eq:ineq_quot_pseudodist}, we conclude that
\(\bar\sfd={\rm dist}_\sfd=\sfd_H\), and thus \(\sfd_H\) is a pseudodistance that induces the topology \(\tau_{\X/\mathcal R}\)
by Lemma \ref{lem:various_top_on_quot}. Since the elements of \(\X/\mathcal R\) are closed subsets of \((\X,\sfd)\), we have
that \(\sfd_H\) is in fact a distance. Finally, assume \((\X,\sfd)\) is complete. Fix a sequence \((a_n)_n\subseteq\X/\mathcal R\)
such that \(\sum_{n=1}^\infty\sfd_H(a_n,a_{n+1})<+\infty\). We can select representatives \(x_n\in a_n\) for every \(n\in\N\)
such that \(\sfd(x_n,x_{n+1})\leq\sfd_H(a_n,a_{n+1})+2^{-n}\). It follows that \(\sum_{n=1}^\infty\sfd(x_n,x_{n+1})<+\infty\),
thus \((x_n)_n\) is a Cauchy sequence in \((\X,\sfd)\). Hence, \(\sfd(x_n,x)\to 0\) for some \(x\in\X\) and thus
\(\sfd_H(a_n,[x]_{\mathcal R})\leq\sfd(x_n,x)\to 0\), proving the completeness of \((\X/\mathcal R,\sfd_H)\).
\end{proof}
\subsection{Metric measure spaces}\label{s:mms}
In this paper, by a \emph{metric measure space} \((\X,\sfd,\mm)\) we mean a complete separable metric space
\((\X,\sfd)\) together with a boundedly-finite Borel measure \(\mm\) on \(\X\) (i.e., \(\mm(B)<+\infty\)
for every bounded Borel set \(B\subseteq\X\)). When \(\mm\in\mathscr P(\X)\), i.e.\ \(\mm\) is a probability
measure, we say that \((\X,\sfd,\mm)\) is a \emph{normalised} metric measure space. Two metric measure spaces
\((\X,\sfd_\X,\mm_\X)\) and \((\Y,\sfd_\Y,\mm_\Y)\) are \emph{isomorphic as metric measure spaces} if there
exists an isometric bijection \(\phi\colon{\rm spt}(\mm_\X)\to{\rm spt}(\mm_\Y)\) such that \(\phi_\#\mm_\X=\mm_\Y\),
where \({\rm spt}(\cdot)\) denotes the \emph{support} of a measure, while \(\phi_\#\mm_\X\) is the
\emph{pushforward measure} of \(\mm_\X\) with respect to \(\phi\). Letting \(\mathcal{X}=[\X,\sfd,\mm]\)
be the isomorphism class of a metric measure space \((\X,\sfd,\mm)\), we define
\[
\mathbb X_1\coloneqq\{[\X,\sfd,\mm]\;|\;(\X,\sfd,\mm)\text{ is a normalised metric measure space}\}.
\]
\begin{remark}\label{rmk:suff_cond_isom_surj}{\rm
If \((\X,\sfd_\X,\mm_\X)\), \((\Y,\sfd_\Y,\mm_\Y)\) are metric measure spaces and \(\phi\colon{\rm spt}(\mm_\X)\to\Y\)
is an isometry with \(\mm_\Y=\phi_\#\mm_\X\), then \(\phi\) is an isomorphism of metric measure spaces. Indeed,
\(\mm_\Y\) is concentrated on \(\phi({\rm spt}(\mm_\X))\), which is closed since \(\phi\) is an isometry,
thus \({\rm spt}(\mm_\Y)\subseteq\phi({\rm spt}(\mm_\X))\). Conversely, if \(y=\phi(x)\in\phi({\rm spt}(\mm_\X))\), then \(\mm_\Y(B_r(y))=\mm_\X(B_r(x))>0\) for every \(r>0\) since \(B_r(x)\cap {\rm spt}(\mm_\X)=\phi^{-1}(B_r(y))\),
thus \(y\in{\rm spt}(\mm_\Y)\). Hence, we have \(\phi({\rm spt}(\mm_\X))={\rm spt}(\mm_\Y)\).
\fr}\end{remark}
\subsection{Gromov's box distance}\label{s:box_dist}
Let us recall Gromov's notion of \emph{box distance} \cite{Gromov07} for normalised spaces;
see also \cite[Chapter 4]{Shioya16}. Given a normalised metric measure space \((\X,\sfd,\mm)\),
a \emph{parameter} of \(\mm\) is a Borel map \(q\colon[0,1)\to\X\) such that \(q_\#\mathscr L_1=\mm\),
where \(\mathscr L_1\) denotes the restriction of the Lebesgue measure to the interval \([0,1)\).
Every normalised metric measure space admits a parameter; see \cite[Lemma 4.2]{Shioya16}. Given two pseudodistances
\(\delta_1\) and \(\delta_2\) on \([0,1)\), we define \(\square(\delta_1,\delta_2)\) as the infimum of
\(\varepsilon\in[0,1]\) for which there is a Borel set \(I_\varepsilon\subseteq[0,1)\) with
\(\mathscr L_1(I_\varepsilon)\geq 1-\varepsilon\) such that
\[
|\delta_1(t,s)-\delta_2(t,s)|\leq\varepsilon\quad\text{ for every }t,s\in I_\varepsilon.
\]
A map \(f\colon\Y\to\X\) between a set \(\Y\) and a metric space \((\X,\sfd)\) induces
a pseudodistance \(f^*\sfd\) on \(\Y\):
\begin{equation}\label{eq:def_pullback_pseudodist}
(f^*\sfd)(x,y)\coloneqq\sfd(f(x),f(y))\quad\text{ for every }x,y\in\Y.
\end{equation}
\begin{definition}[Box distance]
The \emph{box distance} \(\square\colon\mathbb X_1\times\mathbb X_1\to[0,1]\) is defined as
\[
\square([\X,\sfd_\X,\mm_\X],[\Y,\sfd_\Y,\mm_\Y])\coloneqq\inf_{q_\X,q_\Y}\square(q_\X^*\sfd_\X,q_\Y^*\sfd_\Y)
\quad\text{ for every }[\X,\sfd_\X,\mm_\X],[\Y,\sfd_\Y,\mm_\Y]\in\mathbb X_1,
\]
where the infimum is taken over all parameters \(q_\X\) and \(q_\Y\) of \(\mm_\X\) and \(\mm_\Y\), respectively.
\end{definition}

Clearly, the above definition is well posed, i.e.\ it is independent of the chosen representatives.
\subsection{Urysohn's universal space}\label{s:Urysohn_univ_space}
We recall that we use \((\mathbb U,\sfd_{\mathbb U})\) to denote the Urysohn universal metric space \cite{Urysohn27},
which is -- up to isometric bijections -- the unique complete and separable metric space that is both \emph{universal}
and \emph{ultrahomogeneous}. Universality means that every separable metric space can be isometrically embedded
in \(\mathbb U\), whereas ultrahomogeneity means that every isometry \(\phi\colon{\rm F}\to\mathbb U\) defined
on a finite subset \({\rm F}\subseteq\mathbb U\) can be extended to an isometric bijection \(\varphi\colon\mathbb U\to\mathbb U\).
See e.g.\ \cite[Section 3.1]{Hei03} or \cite{Melleray07}, and the references therein, for a more detailed
discussion on \(\mathbb U\).
\medskip

We denote by \({\rm Iso}(\mathbb U)\) the \emph{isometry group} of \(\mathbb U\), i.e.\ the space
of all isometric bijections \(\varphi\colon\mathbb U\to\mathbb U\). Note that \({\rm Iso}(\mathbb U)\)
is a remarkably rich group, due to the ultrahomogeneity of \(\mathbb U\). Furthermore:
\begin{theorem}[Huhunai\v{s}vili \cite{Huhunaisvili55}]\label{thm:ext_Iso_cpt}
\(\mathbb U\) is \emph{compactly ultrahomogeneous}, which means that if \(K\subseteq\mathbb U\)
is a compact set and \(\phi\colon K\to\mathbb U\) is an isometry, then there exists \(\varphi\in{\rm Iso}(\mathbb U)\)
such that \(\varphi|_K=\phi\).
\end{theorem}

Given two normalised metric measure spaces \((\X,\sfd_\X,\mu)\) and \((\Y,\sfd_\Y,\nu)\),
we denote by \({\rm Adm}(\mu;\nu)\) the collection of all \emph{transport plans} \(\pi\) between \(\mu\) to \(\nu\),
i.e.\ of those measures \(\pi\in\mathscr P(\X\times\Y)\) such that \(\pi(A\times\Y)=\mu(A)\)
for all Borel sets \(A\subseteq\X\) and \(\pi(\X\times B)=\nu(B)\) for all Borel sets \(B\subseteq\Y\).
\medskip

For a complete separable metric space \((\X,\sfd)\), the \emph{truncated \(1\)-Wasserstein distance} is given by
\begin{equation}\label{eq:trunc_1-Wasserstein}
W(\mu,\nu)\coloneqq\inf_{\pi\in{\rm Adm}(\mu;\nu)}\int\sfd\wedge 1\,\d\pi
\quad\text{ for every }\mu,\nu\in\mathscr P(\X).
\end{equation}
Then \(W\) is a distance on \(\mathscr P(\X)\) that metrises the weak convergence of probability measures
(see \cite[Corollary 6.13]{Villani09}); by `weak convergence' we mean convergence in duality with bounded continuous functions,
and we use the notation \(\mu_n\rightharpoonup\mu\) to denote weak convergence. \((\mathscr P(\X),W)\) is complete and separable.
We restrict ourselves to truncated Wasserstein distances of exponent \({\rm p}=1\) just for simplicity of presentation;
see Remark \ref{rmk:about_p-Wasserstein} for a brief discussion of the case \({\rm p}\in(1,\infty)\).
\medskip

In the next result, we collect some basic facts about \(W\), including their proofs for convenience.
\begin{lemma}\label{lem:basic_prop_W}
Let \((\X,\sfd)\) be a complete separable metric space. Then the following properties hold:
\begin{itemize}
\item[\(\rm i)\)] If \(\mu\in\mathscr P(\X)\) is supported on a set \(\{x_1,\ldots,x_n\}\subseteq\X\),
then for any \(\varepsilon\in(0,1)\) it holds that
\[
\nu\bigg(\bigcup_{i=1}^n B_\varepsilon(x_i)\bigg)\geq 1-\varepsilon\quad
\text{ for every }\nu\in\mathscr P(\X)\text{ with }W(\mu,\nu)<\varepsilon^2.
\]
\item[\(\rm ii)\)] If \(\mu\in\mathscr P(\X)\) and \(B\subseteq\X\) is a Borel set with \(\mu(B)>0\),
then \(\mu_B\coloneqq\mu(B)^{-1}\mu|_B\) satisfies
\[
W(\mu,\mu_B)\leq 1-\mu(B).
\]
\item[\(\rm iii)\)] If \(\mu\in\mathscr P(\X)\) is concentrated on \(\bigcup_{i=1}^n B_r(x_i)\) for some
\(x_1,\ldots,x_n\in\X\) and \(r>0\), then it holds that \(W(\mu,\sum_{i=1}^n\mu(D_i)\delta_{x_i})\leq r\),
where \(D_i\coloneqq B_r(x_i)\setminus\bigcup_{j<i}B_r(x_j)\) for all \(i=1,\ldots,n\).
\item[\(\rm iv)\)] Given any \(\mu=\sum_{i=1}^n\alpha_i\delta_{x_i}\in\mathscr P(\X)\) and
\(\nu=\sum_{i=1}^n\beta_i\delta_{y_i}\in\mathscr P(\X)\), it holds that
\[
W(\mu,\nu)\leq\max_{i=1,\ldots,n}\sfd(x_i,y_i)\wedge 1+n\max_{i=1,\ldots,n}|\alpha_i-\beta_i|.
\]
\end{itemize}
\end{lemma}
\begin{proof}
\ \\
\(\boldsymbol{{\rm i)}.}\) Take \(\pi\in{\rm Adm}(\mu;\nu)\) with \(\int\sfd\wedge 1\,\d\pi<\varepsilon^2\).
Using \(\{x_1,\ldots,x_n\}\times(\X\setminus\bigcup_{i=1}^n B_\varepsilon(x_i))\subseteq\{\sfd\wedge 1\geq\varepsilon\}\)
and Chebyshev's inequality, we obtain that
\[\begin{split}
1&=\mu(\{x_1,\ldots,x_n\})=\pi(\{x_1,\ldots,x_n\}\times\X)\leq
\pi\bigg(\X\times\bigcup_{i=1}^n B_\varepsilon(x_i)\bigg)+\pi(\{\sfd\wedge 1\geq\varepsilon\})\\
&\leq\nu\bigg(\bigcup_{i=1}^n B_\varepsilon(x_i)\bigg)+\frac{1}{\varepsilon}\int\sfd\wedge 1\,\d\pi
<\nu\bigg(\bigcup_{i=1}^n B_\varepsilon(x_i)\bigg)+\varepsilon.
\end{split}\]
\noindent\(\boldsymbol{{\rm ii)}.}\) Consider
\(\pi\coloneqq({\rm id}_\X,{\rm id}_\X)_\#(\mu|_B)+(\mu|_{\X\setminus B})\otimes\mu_B\in{\rm Adm}(\mu;\mu_B)\).

\noindent\(\boldsymbol{{\rm iii)}.}\) Consider \(\pi\coloneqq\sum_{i=1}^n(\mu|_{D_i})\otimes\delta_{x_i}\in{\rm Adm}(\mu;\tilde\mu)\),
where \(\tilde\mu\coloneqq\sum_{i=1}^n\mu(D_i)\delta_{x_i}\).

\noindent\(\boldsymbol{{\rm iv)}.}\) Consider
\(\pi\coloneqq\sum_{i=1}^n m_i\delta_{(x_i,y_i)}+(1-\sum_{i=1}^n m_i)^{-1}(\sum_{i=1}^n(\alpha_i-m_i)\delta_{x_i})\otimes(\sum_{i=1}^n(\beta_i-m_i)\delta_{y_i})\),
where \(m_i\coloneqq\alpha_i\wedge\beta_i\) for every \(i=1,\ldots,n\). Straightforward
computations yield \(\pi\in{\rm Adm}(\mu;\nu)\).
\end{proof}
\section{The space \texorpdfstring{\(\mathscr P_\sim(\mathbb U)\)}{Psim(U)}}
We equip the space \(\mathscr P(\mathbb U)\) of Borel probability measures on \(\mathbb U\)
with the equivalence relation \(\sim\) defined in \eqref{eq:def_equiv_rel_P(U)}, namely, for any
\(\mu,\nu\in\mathscr P(\mathbb U)\) we declare that
\[
\mu\sim\nu\qquad\Longleftrightarrow\qquad\phi_\#\mu=\nu\;\text{ for some isometric bijection }
\phi\colon{\rm spt}(\mu)\to{\rm spt}(\nu).
\]
We define \(\mathscr P_\sim(\mathbb U)\coloneqq\mathscr P(\mathbb U)/\sim\) and we denote by
\([\mu]\in\mathscr P_\sim(\mathbb U)\) the equivalence class of \(\mu\in\mathscr P(\mathbb U)\).
We equip the space \(\mathscr P_\sim(\mathbb U)\) with the Hausdorff metric \(W_H\) induced by \(W\),
which is defined as in \eqref{eq:def_W_H}. The main goal of this section is to prove the following statement:
\begin{theorem}\label{thm:equiv_W_H}
\((\mathscr P_\sim(\mathbb U),W_H)\) is a complete separable length metric space
whose induced topology coincides with the quotient topology. Moreover, it holds that
\begin{equation}\label{eq:equiv_W_H}
\bar W(\theta,\eta)={\rm dist}_W(\theta,\eta)=W_H(\theta,\eta)\quad\text{ for every }\theta,\eta\in\mathscr P_\sim(\mathbb U),
\end{equation}
where \(\bar W\), \({\rm dist}_W\) and \(W_H\) are defined as in Section \ref{s:quotients_metric_space}.
\end{theorem}

To prove Theorem \ref{thm:equiv_W_H}, we proceed as follows: we show that the isometry group
\({\rm Iso}(\mathbb U)\) acts on \(\mathscr P(\mathbb U)\) by isometries and the closures of its
\({\rm Iso}(\mathbb U)\)-orbits coincide with the elements of \(\mathscr P_\sim(\mathbb U)\) (see Lemma
\ref{lem:Iso_acts_by_isom}). Theorem \ref{thm:equiv_W_H} follows by applying the general principle
presented in Proposition \ref{prop:act_by_isom_conseq}.
\begin{lemma}\label{lem:theta_closed}
Every set \(\theta\in\mathscr P_\sim(\mathbb U)\) is a closed subset of \((\mathscr P(\mathbb U),W)\).
\end{lemma}
\begin{proof}
Let \((\mu_n)_{n\in\N}\subseteq\theta\) and \(\mu\in\mathscr P(\mathbb U)\) satisfy \(W(\mu_n,\mu)\to 0\).
For any \(n\in\N\), take an isometry \(\phi_n\colon{\rm spt}(\mu_1)\to{\rm spt}(\mu_n)\) such that
\((\phi_n)_\#\mu_1=\mu_n\). Since \((\mu_n)_{n\in\N}\) is tight by Prokhorov's theorem, we can find an increasing
sequence \((H_k)_{k\in\N}\) of compact subsets of \(\mathbb U\) such that -- up to a subsequence --
\(\mu_n(\mathbb U\setminus H_k)\leq 2^{-k}\) for every \(k,n\in\N\) with \(k\leq n\).
Define the increasing sequence \((F_k)_{k\in\N}\) as
\[
F_k\coloneqq\bigcap_{n\geq k}\phi_n^{-1}({\rm spt}(\mu_n)\cap H_k)\quad\text{ for every }k\in\N
\]
and \(F\coloneqq\bigcup_{k\in\N}F_k\subseteq{\rm spt}(\mu_1)\). Since
\(\mu_1(\mathbb U\setminus F_k)\leq\sum_{n\geq k}\mu_1(\phi_n^{-1}(\mathbb U\setminus H_k))\leq 2^{-k+1}\)
for every \(k\in\N\), we deduce that \(\mu_1(\mathbb U\setminus F)=0\), thus in particular \(F\) is dense
in \({\rm spt}(\mu_1)\). Fix a dense sequence \((x_k)_{k\in\N}\) in \({\rm spt}(\mu_1)\) such that
\(x_k\in F_k\) for every \(k\in\N\). For any \(k\in\N\) we have that the sequence \((\phi_n(x_k))_{n\geq k}\)
is contained in the compact set \(H_k\), thus a diagonalisation argument ensures that -- up to a further
subsequence -- the limit \(\tilde\phi(x_k)\coloneqq\lim_n\phi_n(x_k)\in H_k\) exists for every \(k\in\N\).
Since each \(\phi_n\) is isometric, the resulting map \(\tilde\phi\colon\{x_k:k\in\N\}\to\mathbb U\) is an isometry.
Hence, \(\tilde\phi\) uniquely extends to an isometry \(\phi\colon{\rm spt}(\mu_1)\to\mathbb U\), which is also
the pointwise limit of \(\phi_n\) as \(n\to\infty\). It follows that \(\mu_n=(\phi_n)_\#\mu_1\rightharpoonup\phi_\#\mu_1\),
so that \(\phi_\#\mu_1=\mu\) and thus \(\mu\in[\mu_1]=\theta\) by Remark \ref{rmk:suff_cond_isom_surj}.
\end{proof}

Note that the isometry group \({\rm Iso}(\mathbb U)\) acts on \(\mathscr P(\mathbb U)\) via pushforward as \(\varphi.\mu\coloneqq\varphi_\#\mu\).
\begin{lemma}\label{lem:Iso_acts_by_isom}
\({\rm Iso}(\mathbb U)\) acts on \(\mathscr P(\mathbb U)\) by isometries. Moreover, the \(\sim\)-equivalence classes
coincide with the closures of the \({\rm Iso}(\mathbb U)\)-orbits in \(\mathscr P(\mathbb U)\).
\end{lemma}
\begin{proof}
For any \(\varphi\in{\rm Iso}(\mathbb U)\), we have that \(\pi\in {\rm Adm}(\mu;\nu)\)
if and only if \((\varphi\times\varphi)_\#\pi\in{\rm Adm}(\varphi_\#\mu;\varphi_\#\nu)\), and
\(\int\sfd_\mathbb U(x,y)\wedge 1\,\d\pi(x,y)=\int\sfd_\mathbb U(x,y)\wedge 1\,\d(\varphi\times\varphi)_\#\pi(x,y)\).
Then \(W(\varphi_\#\mu,\varphi_\#\nu)=W(\mu,\nu)\) for every \(\varphi\in{\rm Iso}(\mathbb U)\)
and \(\mu,\nu\in\mathscr P(\mathbb U)\), i.e.\ \({\rm Iso}(\mathbb U)\) acts on \(\mathscr P(\mathbb U)\)
by isometries. We also claim that
\begin{equation}\label{eq:orbits_Iso}
\mu\sim\nu\qquad\Longleftrightarrow\qquad\mu\in\overline{{\rm Iso}(\mathbb U).\nu}
\end{equation}
holds for any \(\mu,\nu\in\mathscr P(\mathbb U)\); note that \(\overline{{\rm Iso}(\mathbb U).\nu}\) is the closure
of \(\{\varphi_\#\nu:\varphi\in{\rm Iso}(\mathbb U)\}\) in \((\mathscr P(\mathbb U),W)\). To prove the implication
\(\Longrightarrow\), assume \(\mu\sim\nu\). Fix an isometry \(\phi\colon{\rm spt}(\nu)\to{\rm spt}(\mu)\)
with \(\phi_\#\nu=\mu\) and \(\varepsilon>0\). Since \(\nu\) is a Radon measure, we find a compact set
\(K\subseteq{\rm spt}(\nu)\) such that \(\nu(\mathbb U\setminus K)\leq\varepsilon\). By Theorem \ref{thm:ext_Iso_cpt},
there exists \(\varphi\in{\rm Iso}(\mathbb U)\) such that \(\varphi|_K=\phi|_K\). Since
\((\phi,\varphi)_\#\nu\in{\rm Adm}(\mu;\varphi_\#\nu)\),
\[
W(\mu,\varphi_\#\nu)\leq\int\sfd_{\mathbb U}\wedge 1\,\d(\phi,\varphi)_\#\nu
=\int_{\mathbb U\setminus K}\sfd_{\mathbb U}(\phi(z),\varphi(z))\wedge 1\,\d\nu(z)\leq\nu(\mathbb U\setminus K)\leq\varepsilon.
\]
By the arbitrariness of \(\varepsilon\), this shows that \(\mu\in\overline{{\rm Iso}(\mathbb U).\nu}\). To prove the implication
\(\Longleftarrow\), simply note that \({\rm Iso}(\mathbb U).\nu\subseteq[\nu]\), and thus \(\overline{{\rm Iso}(\mathbb U).\nu}\subseteq[\nu]\)
by Lemma \ref{lem:theta_closed}. All in all, \eqref{eq:orbits_Iso} is proven. Since \({\rm Iso}(\mathbb U)\) acts on \(\mathscr P(\mathbb U)\)
by isometries, it is easy to check that \(\mu\in\overline{{\rm Iso}(\mathbb U).\nu}\) if and only if
\(\overline{{\rm Iso}(\mathbb U).\mu}=\overline{{\rm Iso}(\mathbb U).\nu}\). In view of \eqref{eq:orbits_Iso}, we thus conclude
that \([\mu]=\overline{{\rm Iso}(\mathbb U).\mu}\) for every \(\mu\in\mathscr P(\mathbb U)\), as desired.
\end{proof}
\begin{proof}[Proof of Theorem \ref{thm:equiv_W_H}]
It follows from Lemma \ref{lem:Iso_acts_by_isom}, Proposition \ref{prop:act_by_isom_conseq} and the completeness
of \(W\) that \((\mathscr P_\sim(\mathbb U),W_H)\) is a complete metric space whose induced
topology is the quotient topology, and that \eqref{eq:equiv_W_H} holds. Since \((\mathscr P(\mathbb U),W)\)
is separable, we deduce that \((\mathscr P_\sim(\mathbb U),W_H)\) is separable (as quotients of separable
topological spaces are separable). Since \(W\leq 1\), we have that \(W_H\leq 1\), thus accordingly
\((\mathscr P_\sim(\mathbb U),W_H)\) is a metric space. Finally, \((\mathscr P(\mathbb U),W)\) is a geodesic space
(a geodesic between \(\mu,\nu\in\mathscr P(\mathbb U)\) is given by the linear interpolation curve \([0,1]\ni t\mapsto(1-t)\mu+t\nu\),
as easy computations show), whence it readily follows that \(W_H={\rm dist}_W\) is a length distance.
\end{proof}

We remark that, if instead of \(\mathbb U\) we consider any universal complete separable metric space \({\rm Z}\)
(for instance, \({\rm Z}=C([0,1])\) endowed with the supremum norm, which is universal by the Banach--Mazur theorem),
then \({\rm Iso}({\rm Z})\) still acts on \(\mathcal P({\rm Z})\) by isometries. However,
the ultrahomogeneity of \(\mathbb U\) is essential in proving that \(W_H={\rm dist}_W=\bar W\),
and hence that \(W_H\) induces the quotient topology.
\begin{remark}{\rm
We claim that the diameter of \((\mathscr P_\sim(\mathbb U),W_H)\) is equal to \(1\), but it is not attained.

Since \(W_H\leq 1\), the diameter of \((\mathscr P_\sim(\mathbb U),W_H)\) is at most \(1\). To prove that it is equal to \(1\),
take \((p_n)_{n\in\N}\subseteq\mathbb U\) so that \(\sfd_{\mathbb U}(p_n,p_m)=2\) for all \(n\neq m\). Any given point
\(p\in\mathbb U\) satisfies \(\sfd_{\mathbb U}(p,p_n)\geq 1\) for all but possibly one index \(n\in\N\). Since
\(W_H([\delta_{p_1}],[n^{-1}\sum_{i=1}^n\delta_{p_i}])=\frac{1}{n}\inf_{p\in\mathbb U}\sum_{i=1}^n\sfd_{\mathbb U}(p,p_i)\wedge 1\),
we have \(W_H([\delta_{p_1}],[n^{-1}\sum_{i=1}^n\delta_{p_i}])\geq\frac{n-1}{n}\to 1\) as \(n\to\infty\).
Finally, let us prove that \(W_H(\theta,\eta)<1\) for every \(\theta,\eta\in\mathscr P_\sim(\mathbb U)\). More precisely,
for any \(\mu,\nu\in\mathscr P(\mathbb U)\) it holds that
\begin{equation}\label{eq:upper_bound_W_H}
W_H([\mu],[\nu])\leq 1-\sup_{0<r<1/2}(1-2r)\sup_{p\in{\rm spt}(\mu)}\sup_{q\in{\rm spt}(\nu)}\min\{\mu(B_r(p)),\nu(B_r(q))\}.
\end{equation}
To prove it, fix any \(r\in(0,1/2)\), \(p\in{\rm spt}(\mu)\) and \(q\in{\rm spt}(\nu)\). Take a map \(\varphi\in{\rm Iso}(\mathbb U)\)
with \(\varphi(q)=p\) and set \(\bar\nu\coloneqq\varphi_\#\nu\). Letting \(\alpha\coloneqq\mu(B_r(p))\),
\(\beta\coloneqq\bar\nu(B_r(p))=\nu(B_r(q))\) and \(\varepsilon\coloneqq\min\{\alpha,\beta\}\), we define
\[
\pi\coloneqq\frac{\varepsilon}{\alpha\beta}(\mu|_{B_r(p)})\otimes(\bar\nu|_{B_r(p)})
+\frac{1}{(1-\varepsilon)\alpha\beta}(\alpha\mu-\varepsilon\mu|_{B_r(p)})\otimes
(\beta\bar\nu-\varepsilon\bar\nu|_{B_r(p)}).
\]
Given that \(\sfd_{\mathbb U}(x,y)<2r\) for every \(x,y\in B_r(p)\) and \(\pi\in{\rm Adm}(\mu;\bar\nu)\),
we conclude that
\[
W_H([\mu],[\nu])\leq W(\mu,\bar\nu)
\leq\frac{\varepsilon}{\alpha\beta}2r\,\mu(B_r(p))\bar\nu(B_r(p))+(1-\varepsilon)=1-(1-2r)\varepsilon,
\]
whence \eqref{eq:upper_bound_W_H} follows by the arbitrariness of \(r\), \(p\) and \(q\).
\fr}\end{remark}

There is an active line of research devoted to the study of the isometry groups of Wasserstein spaces,
see, e.g., the works 
\cite{Kloeckner2010,GeherTitkosVirosztek2019,GeherTitkosVirosztek2022,SantosRodriguez22,GeherHruskovaTitkosVirosztek2025},
to name but a few. Every isometry of a metric space induces -- by pushforward -- an isometry of its Wasserstein space.
A fundamental question is whether all Wasserstein isometries arise in this way. This leads to a dichotomy between
\emph{isometrically rigid} Wasserstein spaces, whose isometries are all induced by isometries of the underlying metric
space, and \emph{flexible} ones, which admit other isometries as well. It would be interesting to investigate such
questions when the underlying metric space is the Urysohn universal space \(\mathbb U\).
\section{The measured-Gromov convergence}
In analogy with the notion of pmG-convergence developed in \cite{GigliMondinoSavare15}, we give the following definition:
\begin{definition}[mG convergence]\label{def:mG}
Let \(([\X_n,\sfd_n,\mm_n])_{n\in\bar\N}\subseteq\mathbb X_1\) be given, where \(\bar\N\coloneqq\N\cup\{\infty\}\).
Then we say that \([\X_n,\sfd_n,\mm_n]\) \emph{mG-converges} to \([\X_\infty,\sfd_\infty,\mm_\infty]\) if
there exist a complete separable metric space \((\Y,\sfd_\Y)\) and isometries \(\iota_n\colon\X_n\to\Y\)
such that \((\iota_n)_\#\mm_n\) weakly converges to \((\iota_\infty)_\#\mm_\infty\).
\end{definition}
\begin{definition}\label{def:Phi}
We define the map \(\Phi\colon\mathbb X_1\to\mathscr P_\sim(\mathbb U)\) as follows: for any \([\X,\sfd,\mm]\in\mathbb X_1\), we set
\[
\Phi([\X,\sfd,\mm])\coloneqq[f_\#\mm]\quad\text{ for any isometry }f\colon\X\to\mathbb U.
\]
\end{definition}
Taking Remark \ref{rmk:suff_cond_isom_surj} into account, it is straightforward to check that the map
\(\Phi\) is well defined and bijective.
By applying Theorem \ref{thm:equiv_W_H}, we then deduce that
\begin{equation}\label{eq:def_dmG}
\sfd_{\rm mG}\coloneqq\Phi^*W_H\colon\mathbb X_1\times\mathbb X_1\to[0,1]
\end{equation}
is a complete and separable distance on \(\mathbb X_1\); recall from \eqref{eq:def_pullback_pseudodist} the
definition of \(\Phi^*W_H\).
\begin{theorem}\label{thm:equiv_mG-conv}
Let \((\mathcal X_n)_{n\in\bar\N}\subseteq\mathbb X_1\) be given. Then it holds that \(\mathcal X_n\) mG-converges
to \(\mathcal X_\infty\) if and only if \(\sfd_{\rm mG}(\mathcal X_n,\mathcal X_\infty)\to 0\).
\end{theorem}
\begin{proof}
Assume \([\X_n,\sfd_n,\mm_n]\to[\X_\infty,\sfd_\infty,\mm_\infty]\) in the mG sense.
We can find isometries \(\iota_n\colon\X_n\to\mathbb U\) such that \((\iota_n)_\#\mm_n\) weakly
converges to \((\iota_\infty)_\#\mm_\infty\). Then \eqref{eq:equiv_W_H} yields
\[
\sfd_{\rm mG}([\X_n,\sfd_n,\mm_n],[\X_\infty,\sfd_\infty,\mm_\infty])\leq
W((\iota_n)_\#\mm_n,(\iota_\infty)_\#\mm_\infty)\to 0\quad\text{ as }n\to\infty.
\]
On the contrary, assume \(\sfd_{\rm mG}([\X_n,\sfd_n,\mm_n],[\X_\infty,\sfd_\infty,\mm_\infty])\to 0\).
Denote \(\theta_n\coloneqq\Phi([\X_n,\sfd_n,\mm_n])\) for every \(n\in\bar\N\) and fix some representative
\(\mu_\infty\in\theta_\infty\). For any \(n\in\N\), we find \(\mu_n\in\theta_n\) such that
\(W(\mu_n,\mu_\infty)\leq W_H(\theta_n,\theta_\infty)+1/n\), thus \(W(\mu_n,\mu_\infty)\to 0\).
Since \([{\rm spt}(\mu_n),\sfd_{\mathbb U},\mu_n]=[\X_n,\sfd_n,\mm_n]\) for every \(n\in\bar\N\),
we conclude that \([\X_n,\sfd_n,\mm_n]\to[\X_\infty,\sfd_\infty,\mm_\infty]\) in the mG sense.
\end{proof}

The mG-topology (i.e.\ the topology induced by \(\sfd_{\rm mG}\)) coincides with the topology induced by the
Gromov--Prokhorov metric as a consequence of \cite[Lemma 5.8]{GPW2009}, thus also with the Gromov-weak topology
by \cite[Section 9]{GPW2009}.
\begin{remark}\label{rmk:GW-dmG}{\rm
For any \(\mathcal X=[\X,\sfd_\X,\mm_\X]\in\mathbb X_1\) and \(\mathcal Y=[\Y,\sfd_\Y,\mm_\Y]\in\mathbb X_1\),
we define
\[
{\rm G}_W(\mathcal X,\mathcal Y)\coloneqq\inf_{({\rm Z},\sfd_{\rm Z}),\iota_\X,\iota_\Y}\inf\bigg\{
\int\sfd_{\rm Z}\wedge 1\,\d\pi\;\bigg|\;\pi\in{\rm Adm}((\iota_\X)_\#\mm_\X;(\iota_\Y)_\#\mm_\Y)\bigg\},
\]
where the first infimum is taken over all complete separable metric spaces \(({\rm Z},\sfd_{\rm Z})\)
and all isometric embeddings \(\iota_\X\colon\X\to{\rm Z}\) and \(\iota_\Y\colon\Y\to{\rm Z}\).
Note that \({\rm G}_W(\mathcal X,\mathcal Y)\) coincides with the quantity \(\mathbb G_{\rm W}(\mathcal X,\mathcal Y)\)
that is defined at the beginning of \cite[Section 3.2.4]{GigliMondinoSavare15} (with \(c(t)\coloneqq t\wedge 1\)).
We claim that
\[
{\rm G}_W(\mathcal X,\mathcal Y)=\sfd_{\rm mG}(\mathcal X,\mathcal Y)\quad\text{ for every }\mathcal X,\mathcal Y\in\mathbb X_1.
\]
To obtain \(\leq\), just observe that \({\rm G}_W(\mathcal X,\mathcal Y)\leq\int\sfd_{\mathbb U}(x,y)\wedge 1\,\d\pi(x,y)\)
for every \(\mu\in\Phi(\mathcal X)\), \(\nu\in\Phi(\mathcal Y)\) and \(\pi\in{\rm Adm}(\mu;\nu)\).
The converse inequality \(\geq\) readily follows from the universality of \(\mathbb U\).
\fr}\end{remark}
\subsection{Compactness criterion for the mG-convergence}
Our next result, whose proof is rather straightforward, provides a characterisation of the precompact subsets of
\((\mathscr P_\sim(\mathbb U),W_H)\).
\begin{theorem}\label{thm:precpt_W_H}
Let \(\mathcal K\subseteq\mathscr P_\sim(\mathbb U)\) be given. Then the following conditions are equivalent:
\begin{itemize}
\item[\(\rm i)\)] The set \(\mathcal K\) is \(W_H\)-precompact (or, equivalently, the set \(\Phi^{-1}(\mathcal K)\)
is \(\sfd_{\rm mG}\)-precompact).
\item[\(\rm ii)\)] Given any \(\varepsilon>0\), we can find points \(x_1,\ldots,x_N\in\mathbb U\) with
\({\rm diam}\{x_1,\ldots,x_N\}\leq N\) having the property that for any \(\theta\in\mathcal K\) there exists \(\mu\in\theta\)
such that \(\mu(\bigcup_{i=1}^N B_\varepsilon(x_i))>1-\varepsilon\).
\item[\(\rm iii)\)] Given any \(\varepsilon>0\), there exists \(N\in\N\) such that for any
\(\mu\in\mathscr P(\mathbb U)\) with \([\mu]\in\mathcal K\) we can find points
\(x^\mu_1,\ldots,x^\mu_N\in{\rm spt}(\mu)\) such that \(\mu(\bigcup_{i=1}^N B_\varepsilon(x^\mu_i))>1-\varepsilon\)
and \({\rm diam}\{x^\mu_1,\ldots,x^\mu_N\}\leq N\).
\end{itemize}
\end{theorem}
\begin{proof}
\ \\
\(\boldsymbol{{\rm i)}\Longrightarrow{\rm ii)}.}\) Assume \(\mathcal K\) is \(W_H\)-precompact and fix \(\varepsilon\in(0,1)\).
Then we find an \(\varepsilon^2\)-net \(\{[\mu_1],\ldots,[\mu_k]\}\) for \(\mathcal K\) in \((\mathscr P_\sim(\mathbb U),W_H)\).
Since finitely-supported probability measures are weakly dense in \(\mathscr P(\mathbb U)\) and \(W_H([\mu],[\nu])\leq W(\mu,\nu)\),
we can assume that for some finite set \(F=\{x_1,\ldots,x_N\}\subseteq\mathbb U\) it holds that \(\mu_j\) is concentrated on \(F\)
for every \(j=1,\ldots,k\). Without loss of generality, we can also assume that \({\rm diam}\{x_1,\ldots,x_N\}\leq N\). Now, fix
\(\theta\in\mathcal K\). There exists \(j=1,\ldots,k\) such that \(W_H(\theta,[\mu_j])<\varepsilon^2\), thus
\(W(\mu,\mu_j)<\varepsilon^2\) for some \(\mu\in\theta\). By Lemma \ref{lem:basic_prop_W} i), we conclude that
\(\mu(\bigcup_{i=1}^N B_\varepsilon(x_i))\geq 1-\varepsilon\).
\smallskip

\noindent\(\boldsymbol{{\rm ii)}\Longrightarrow{\rm iii)}.}\) Straightforward.
\smallskip

\noindent\(\boldsymbol{{\rm iii)}\Longrightarrow{\rm i)}.}\) Assume iii) holds and fix \(\varepsilon>0\). Pick any
\(\delta\in(0,1)\) so that \(7\delta<\varepsilon\). For any \(\theta\in\mathcal K\), take \(\mu_\theta\in\theta\)
and \(x_1^\theta,\ldots,x_N^\theta\in{\rm spt}(\mu_\theta)\) such that
\(\mu_\theta(\bigcup_{i=1}^N B_\delta(x_i^\theta))\geq 1-\delta\) and \({\rm diam}\{x_1^\theta,\ldots,x_N^\theta\}\leq N\), for some \(N\in\N\)
that is independent of \(\theta\). Letting \(D_i^\theta\coloneqq B_\delta(x_i^\theta)\setminus\bigcup_{j<i}B_\delta(x_j^\theta)\), we define
\[
\hat\mu_\theta\coloneqq\sum_{i=1}^N\beta_i(\theta)\delta_{x_i^\theta}\in\mathscr P(\mathbb U)\quad\text{ for every }
\theta\in\mathcal K\text{, where }\beta_i(\theta)\coloneqq\frac{\mu_\theta(D_i^\theta)}{\mu_\theta(\bigcup_{j=1}^N B_\delta(x_j^\theta)).}
\]
Lemma \ref{lem:basic_prop_W} ii) and iii) yield \(W(\hat\mu_\theta,\mu_\theta)\leq 2\delta\).
Define \(\alpha_{i,j}(\theta)\coloneqq\sfd_{\mathbb U}(x_i^\theta,x_j^\theta)\) for every \(1\leq i<j\leq N\) and \(\theta\in\mathcal K\).
We then define the map \(\psi\colon\mathcal K\to[0,N]^{N(N-1)/2}\times[0,1/(1-\delta)]^N\) as
\[
\psi(\theta)\coloneqq(\alpha_{1,2}(\theta),\alpha_{1,3}(\theta),\ldots,\alpha_{N-1,N}(\theta),\beta_1(\theta),\ldots,\beta_N(\theta))
\quad\text{ for every }\theta\in\mathcal K.
\]
Take \(\theta_1,\ldots,\theta_k\in\mathcal K\) so that \(\{\psi(\theta_1),\ldots,\psi(\theta_k)\}\) is a
\(\frac{\delta}{N}\)-net in \((\R^{N(N+1)/2},\sfd_{\rm sup})\) for \(\psi(\mathcal K)\), where \(\sfd_{\rm sup}\)
is the supremum distance. Fix an isometry \(f\colon(\R^N,\sfd_{\rm sup})\to\mathbb U\). For any
\(\theta\in\mathcal K\), we define
\[
\iota_\theta(x)\coloneqq f(\sfd_{\mathbb U}(x,x_1^\theta)-\sfd_{\mathbb U}(x^\theta_1,x_1^\theta),
\ldots,\sfd_{\mathbb U}(x,x_N^\theta)-\sfd_{\mathbb U}(x^\theta_1,x_N^\theta))\in\mathbb U\quad
\text{ for every }x\in\{x_1^\theta,\ldots,x_N^\theta\}.
\]
Note that the resulting map \(\iota_\theta\colon(\{x_1^\theta,\ldots,x_N^\theta\},\sfd_{\mathbb U})\to\mathbb U\) is an isometry.
By the ultrahomogeneity of \(\mathbb U\), we can extend each \(\iota_\theta\) to a map \(\varphi^\theta\in{\rm Iso}(\mathbb U)\).
Now, fix any \(\theta\in\mathcal K\) and choose \(\ell=1,\ldots,k\) for which \(\sfd_{\rm sup}(\psi(\theta),\psi(\theta_\ell))<\delta\).
Observe that for any \(i=1,\ldots,N\) we have \(|\beta_i(\theta)-\beta_i(\theta_\ell)|<\frac{\delta}{N}\) and
\[
\sfd_{\mathbb U}(\varphi^\theta(x_i^\theta),\varphi^{\theta_\ell}(x_i^{\theta_\ell}))
=\max_{1\leq j\leq N}|\alpha_{i,j}(\theta)-\alpha_{1,j}(\theta)-(\alpha_{i,j}(\theta_\ell)-\alpha_{1,j}(\theta_\ell))|
\leq 2\,\sfd_{\rm sup}(\psi(\theta),\psi(\theta_\ell))<2\delta.
\]
Applying Lemma \ref{lem:basic_prop_W} iv) to \(\varphi^\theta_\#\hat\mu_\theta\) and \(\varphi^{\theta_\ell}_\#\hat\mu_{\theta_\ell}\),
we thus obtain \(W(\varphi^\theta_\#\hat\mu_\theta,\varphi^{\theta_\ell}_\#\hat\mu_{\theta_\ell})\leq 3\delta\). Therefore,
\[\begin{split}
W_H(\theta,\theta_\ell)&\leq W(\varphi^\theta_\#\mu_\theta,\varphi^{\theta_\ell}_\#\mu_{\theta_\ell})
\leq W(\varphi^\theta_\#\mu_\theta,\varphi^\theta_\#\hat\mu_\theta)
+W(\varphi^\theta_\#\hat\mu_\theta,\varphi^{\theta_\ell}_\#\hat\mu_{\theta_\ell})
+W(\varphi^{\theta_\ell}_\#\hat\mu_{\theta_\ell},\varphi^{\theta_\ell}_\#\mu_{\theta_\ell})\\
&=W(\mu_\theta,\hat\mu_\theta)+W(\varphi^\theta_\#\hat\mu_\theta,\varphi^{\theta_\ell}_\#\hat\mu_{\theta_\ell})+W(\hat\mu_{\theta_\ell},\mu_{\theta_\ell})
\leq 7\delta<\varepsilon.
\end{split}\]
This proves that \(\{\theta_1,\ldots,\theta_k\}\) is an \(\varepsilon\)-net for \(\mathcal K\) in \((\mathscr P_\sim(\mathbb U),W_H)\),
so that \(\mathcal K\) is \(W_H\)-precompact.
\end{proof}
\section{Comparison with the box topology}
Given metric measure spaces \((\X,\sfd_\X,\mm_\X)\) and \((\Y,\sfd_\Y,\mm_\Y)\), we denote by
\({\rm Cpl}(\sfd_\X,\mm_\X;\sfd_\Y,\mm_\Y)\) the collection of all those complete and separable distances
\(\varrho\) on \({\rm spt}(\mm_\X)\sqcup{
\rm spt}(\mm_\Y)\) that are \emph{couplings} of \((\sfd_\X,\mm_\X)\) and \((\sfd_\Y,\mm_\Y)\),
i.e.\ that satisfy \(\varrho(x,\tilde x)=\sfd_\X(x,\tilde x)\) for all \(x,\tilde x\in{\rm spt}(\mm_\X)\)
and \(\varrho(y,\tilde y)=\sfd_\Y(y,\tilde y)\) for all \(y,\tilde y\in{\rm spt}(\mm_\Y)\).
We define the truncated version of Sturm's \(L_1\)-transportation distance \cite{Sturm2006a} as
\begin{equation}\label{eq:def_D}
{\rm D}(\mathcal X,\mathcal Y)\coloneqq\inf\bigg\{\int\varrho\wedge 1\,\d\ppi\;\big|\;
\varrho\in{\rm Cpl}(\sfd_\X,\mm_\X;\sfd_\Y,\mm_\Y),\,\pi\in{\rm Adm}(\mm_\X;\mm_\Y)\bigg\}
\end{equation}
for every \(\mathcal X=[\X,\sfd_\X,\mm_\X]\in\mathbb X_1\) and \(\mathcal Y=[\Y,\sfd_\Y,\mm_\Y]\in\mathbb X_1\).
Note that \({\rm D}(\mathcal X,\mathcal Y)\) is well defined, since that the right-hand side of
\eqref{eq:def_D} does not depend on the chosen representatives of \(\mathcal X\) and \(\mathcal Y\).
\begin{proposition}\label{prop:D=d_mG}
It holds that \({\rm D}(\mathcal X,\mathcal Y)=\sfd_{\rm mG}(\mathcal X,\mathcal Y)\) for every \(\mathcal X,\mathcal Y\in\mathbb X_1\).
\end{proposition}
\begin{proof}
Fix any \(\mathcal X=[\X,\sfd_\X,\mm_\X]\in\mathbb X_1\) and \(\mathcal Y=[\Y,\sfd_\Y,\mm_\Y]\in\mathbb X_1\).
Let \(\varrho\in{\rm Cpl}(\sfd_\X,\mm_\X;\sfd_\Y,\mm_\Y)\) and \(\pi\in{\rm Adm}(\mm_\X;\mm_\Y)\) be given.
Fix some isometry \(\iota\colon(\X\sqcup\Y,\varrho)\to(\mathbb U,\sfd_{\mathbb U})\). Let
\(i_\X\colon\X\to\X\sqcup\Y\) and \(i_\Y\colon\Y\to\X\sqcup\Y\) be the inclusion maps. Define
\(\bar\pi\coloneqq(\iota\circ i_\X,\iota\circ i_\Y)_\#\pi\in{\rm Adm}((\iota\circ i_\X)_\#\mm_\X;(\iota\circ i_\Y)_\#\mm_\Y)\).
Since \((\iota\circ i_\X)_\#\mm_\X\in\Phi(\mathcal X)\) and \((\iota\circ i_\Y)_\#\mm_\Y\in\Phi(\mathcal Y)\),
we deduce that
\[
\sfd_{\rm mG}(\mathcal X,\mathcal Y)\leq W((\iota\circ i_\X)_\#\mm_\X,(\iota\circ i_\Y)_\#\mm_\Y)
\leq\int\sfd_{\mathbb U}\wedge 1\,\d\bar\pi=\int\varrho\wedge 1\,\d\pi.
\]
Thanks to the arbitrariness of \(\varrho\) and \(\pi\), we can conclude that \(\sfd_{\rm mG}(\mathcal X,\mathcal Y)\leq{\rm D}(\mathcal X,\mathcal Y)\).

To prove the converse inequality, fix \(\varepsilon>0\). We can find \(\mu\in\Phi(\mathcal X)\),
\(\nu\in\Phi(\mathcal Y)\) and \(\pi\in{\rm Adm}(\mu;\nu)\) such that
\(\int\sfd_{\mathbb U}\wedge 1\,\d\pi<\sfd_{\rm mG}(\mathcal X,\mathcal Y)+\varepsilon\).
We define the distance \(\varrho_\varepsilon\) on \(\mathbb U\sqcup\mathbb U\) as
\[
\varrho_\varepsilon(x,y)\coloneqq\left\{\begin{array}{ll}
\sfd_{\mathbb U}(x,y)\\
\varepsilon+\sfd_{\mathbb U}(x,y)\\
\end{array}\quad\begin{array}{ll}
\text{ if }x\text{ and }y\text{ belong to the same copy of }\mathbb U\text{ in }\mathbb U\sqcup\mathbb U,\\
\text{ otherwise.}
\end{array}\right.
\]
Then we have that \(\varrho_\varepsilon\in{\rm Cpl}(\sfd_{\mathbb U},\mu;\sfd_{\mathbb U},\nu)\), so that accordingly
\[
{\rm D}(\mathcal X,\mathcal Y)\leq\int\varrho_\varepsilon\wedge 1\,\d\pi\leq
\varepsilon+\int\sfd_{\mathbb U}\wedge 1\,\d\pi<\sfd_{\rm mG}(\mathcal X,\mathcal Y)+2\varepsilon.
\]
Letting \(\varepsilon\searrow 0\), we conclude that
\({\rm D}(\mathcal X,\mathcal Y)\leq\sfd_{\rm mG}(\mathcal X,\mathcal Y)\).
Therefore, the statement is achieved.
\end{proof}
\begin{theorem}\label{thm:box_equiv_mG}
It holds that
\begin{equation}\label{eq:equiv_box_dist}
\frac{1}{2}\,\square(\mathcal X,\mathcal Y)^2\leq\sfd_{\rm mG}(\mathcal X,\mathcal Y)\leq\frac{3}{2}\square(\mathcal X,\mathcal Y)
\quad\text{ for every }\mathcal X,\mathcal Y\in\mathbb X_1.
\end{equation}
\end{theorem}
\begin{proof}
Fix any \(\mathcal X,\mathcal Y\in\mathbb X\) and \(\varepsilon\in(0,1]\) such that
\(\sfd_{\rm mG}(\mathcal X,\mathcal Y)<\varepsilon\). Then we have \(W(\mu,\nu)<\varepsilon\)
for some \(\mu\in\Phi(\mathcal X)\) and \(\nu\in\Phi(\mathcal Y)\), so that
\(\int\sfd_{\mathbb U}\wedge 1\,\d\pi<\varepsilon\) for some \(\pi\in{\rm Adm}(\mu;\nu)\).
Fix a parameter \(Q=(q_1,q_2)\colon[0,1)\to\mathbb U\times\mathbb U\) of \(\pi\), and let
\(I_\varepsilon\coloneqq Q^{-1}(\{(x,y)\in\mathbb U\times\mathbb U:\sfd_{\mathbb U}(x,y)<\sqrt{\varepsilon/2}\})\).
Then \(\mathscr L_1([0,1)\setminus I_\varepsilon)=\pi(\{(x,y)\in\mathbb U\times\mathbb U:\sfd_{\mathbb U}(x,y)\wedge 1\geq\sqrt{\varepsilon/2}\}\})\leq\sqrt{2\varepsilon}\)
by Chebyshev's inequality, and \(\sfd_{\mathbb U}(q_1(t),q_2(t))<\sqrt{\varepsilon/2}\) for every
\(t\in I_\varepsilon\). In particular, we can estimate
\[
|\sfd_{\mathbb U}(q_1(t),q_1(s))-\sfd_{\mathbb U}(q_2(t),q_2(s))|\leq
\sfd_{\mathbb U}(q_1(t),q_2(t))+\sfd_{\mathbb U}(q_1(s),q_2(s))<\sqrt{2\varepsilon}
\quad\text{ for all }t,s\in I_\varepsilon.
\]
Since \(q_1\) and \(q_2\) are parameters of \(\mu\) and \(\nu\), respectively, we deduce that
\(\square(\mathcal X,\mathcal Y)\leq\sqrt{2\varepsilon}\). Letting \(\varepsilon\searrow\sfd_{\rm mG}(\mathcal X,\mathcal Y)\),
we conclude that \(\square(\mathcal X,\mathcal Y)\leq\sqrt{2\,\sfd_{\rm mG}(\mathcal X,\mathcal Y)}\),
namely the first inequality in \eqref{eq:equiv_box_dist}.

Next, fix any \(\mathcal X,\mathcal Y\in\mathbb X\) and \(\varepsilon\in(0,1]\) such that
\(\square(\mathcal X,\mathcal Y)<\varepsilon\). Take \(\mu\in\Phi(\mathcal X)\) and
\(\nu\in\Phi(\mathcal Y)\). Then there exist parameters \(q_1\) and \(q_2\) of \(\mu\) and \(\nu\), respectively,
and a Borel set \(I_\varepsilon\subseteq[0,1)\) such that \(\mathscr L_1(I_\varepsilon)>1-\varepsilon\) and
\(|\sfd_{\mathbb U}(q_1(t),q_1(s))-\sfd_{\mathbb U}(q_2(t),q_2(s))|<\varepsilon\) for every \(t,s\in I_\varepsilon\).
The preceding condition on $I_\varepsilon$ ensures the triangle inequality, allowing us to equip the set \(V\coloneqq{\rm spt}(\mu)\sqcup{\rm spt}(\nu)\) with the complete and
separable distance \(\varrho\), which we define as
\[
\varrho(x,y)\coloneqq\left\{\begin{array}{lll}
\sfd_{\mathbb U}(x,y)\\
\frac{\varepsilon}{2}+\inf_{s\in I_\varepsilon}[\sfd_{\mathbb U}(x,q_1(s))+\sfd_{\mathbb U}(y,q_2(s))]\\
\frac{\varepsilon}{2}+\inf_{s\in I_\varepsilon}[\sfd_{\mathbb U}(x,q_2(s))+\sfd_{\mathbb U}(y,q_1(s))]
\end{array}\quad\begin{array}{lll}
\text{ if }x,y\in{\rm spt}(\mu)\text{ or }x,y\in{\rm spt}(\nu),\\
\text{ if }x\in{\rm spt}(\mu)\text{ and }y\in{\rm spt}(\nu),\\
\text{ if }x\in{\rm spt}(\nu)\text{ and }y\in{\rm spt}(\mu).
\end{array}\right.
\]
Applying Proposition \ref{prop:D=d_mG} and using the fact \((q_1,q_2)_\#\mathscr L_1\in{\rm Adm}(\mu;\nu)\),
we deduce that
\[\begin{split}
\sfd_{\rm mG}(\mathcal X,\mathcal Y)&={\rm D}(\mathcal X,\mathcal Y)\leq
\int\varrho\wedge 1\,\d(q_1,q_2)_\#\mathscr L_1
\leq\mathscr L_1([0,1)\setminus I_\varepsilon)+\int_{I_\varepsilon}\varrho(q_1(t),q_2(t))\wedge 1\,\d\mathscr L_1(t)\\
&<\varepsilon+\frac{\varepsilon}{2}\mathscr L_1(I_\varepsilon)<\frac{3}{2}\varepsilon,
\end{split}\]
where we used that fact that \(\varrho(q_1(t),q_2(t))=\frac{\varepsilon}{2}\) for all \(t\in I_\varepsilon\).
Letting \(\varepsilon\searrow\square(\mathcal X,\mathcal Y)\), we conclude that
\(\sfd_{\rm mG}(\mathcal X,\mathcal Y)\leq\frac{3}{2}\,\square(\mathcal X,\mathcal Y)\),
namely the second inequality in \eqref{eq:equiv_box_dist}. This completes the proof.
\end{proof}
\begin{remark}{\rm
Theorem \ref{thm:box_equiv_mG} implies that the mG-topology coincides with the box topology;
this follows also from the fact that the Gromov--Prokhorov metric and the box distance are Lipschitz equivalent,
as it was proven in \cite{Lohr2013}.
The topological properties of \((\mathbb X_1,\square)\) have been investigated in the works
\cite{KNSh2024,KazukawaNakajimaShioya24}. We expect that the H\"{o}lder equivalence between \(\square\)
and \(\sfd_{\rm mG}\) established in Theorem \ref{thm:box_equiv_mG} may prove useful for further
understanding the properties of \(\square\), although we do not pursue this research direction
in the present work.
\fr}\end{remark}

Combining Theorems \ref{thm:precpt_W_H} and \ref{thm:box_equiv_mG}, we recover the \(\square\)-compactness
criterion \cite[Lemma 4.28]{Shioya16}:
\begin{corollary}[Compactness criterion for \(\square\)]\label{cor:compact_box_top}
Let \(\mathcal F=\{[\X_j,\sfd_j,\mm_j]\}_{j\in J}\) be a given subset of \(\mathbb X_1\). Then \(\mathcal F\)
is \(\square\)-precompact if and only if for any \(\varepsilon>0\) there exist \(N\in\N\) and \(x^j_1,\ldots,x^j_N\in\X_j\)
for any \(j\in J\) such that \(\mm_j(\bigcup_{i=1}^N B_\varepsilon(x^j_i))>1-\varepsilon\)
and \({\rm diam}\{x^j_1,\ldots,x^j_N\}\leq N\).
\end{corollary}
\section{Relation with the pmG-convergence}
By a \emph{pointed metric measure space} \((\X,\sfd,\mm,p)\) we mean a metric measure space
\((\X,\sfd,\mm)\) together with a distinguished point \(p\in{\rm spt}(\mm)\), and when \(\mm\in\mathscr P(\X)\)
we say that \((\X,\sfd,\mm,p)\) is \emph{normalised}. Two pointed metric measure spaces
\((\X,\sfd_\X,\mm_\X,p_\X)\) and \((\Y,\sfd_\Y,\mm_\Y,p_\Y)\) are \emph{isomorphic as pointed metric measure spaces}
if there exists an isomorphism of metric measure spaces \(\phi\) from \((\X,\sfd_\X,\mm_\X)\) to
\((\Y,\sfd_\Y,\mm_\Y)\) satisfying \(\phi(p_\X)=p_\Y\). Letting \([\X,\sfd,\mm,p]\) be the isomorphism class
of a pointed metric measure space \((\X,\sfd,\mm,p)\), we define
\[\begin{split}
\mathbb X_{pt}&\coloneqq\{[\X,\sfd,\mm,p]\;|\;(\X,\sfd,\mm,p)\text{ is a pointed metric measure space}\},\\
\mathbb X_{1,pt}&\coloneqq\{[\X,\sfd,\mm,p]\in\mathbb X_{pt}\;|\;[\X,\sfd,\mm]\in\mathbb X_1\}.
\end{split}\]

Let us recall the notion of pointed-measured-Gromov convergence, in its extrinsic formulation, introduced by
Gigli, Mondino and Savar\'{e} in \cite{GigliMondinoSavare15}. Let \(([\X_n,\sfd_n,\mm_n,p_n])_{n\in\bar\N}\subseteq\mathbb X_{pt}\)
be given. We say that \([\X_n,\sfd_n,\mm_n,p_n]\) \emph{pmG-converges} to \([\X_\infty,\sfd_\infty,\mm_\infty,p_\infty]\)
if there exist a complete separable metric space \((\Y,\sfd_\Y)\) and isometries \(\iota_n\colon\X_n\to\Y\) such that
\(\iota_n(p_n)\to\iota_\infty(p_\infty)\) and
\begin{equation}\label{eq:def_pmG}
\int f\,\d(\iota_n)_\#\mm_n\to\int f\,\d(\iota_\infty)_\#\mm_\infty\quad
\text{ for every }f\in C_b(\Y)\text{ with bounded support.}
\end{equation}
We recall that if \(\mm_n\in\mathscr P(\X_n)\) for every \(n\in\bar\N\), then \eqref{eq:def_pmG} holds if and only
if \((\iota_n)_\#\mm_n\rightharpoonup(\iota_\infty)_\#\mm_\infty\).
Two equivalent notions are the \emph{wpmGH-convergence}, introduced by the second and third named
authors in \cite{PS2021}, and the convergence induced by the distance \(d_*\), considered by Bate in \cite{Bate22}.
\begin{proposition}\label{prop:embd_Xpt_to_X}
The map \({\sf j}\colon(\mathbb X_{1,pt},\tau_{\rm pmG})\to(\mathbb X_1,\square)\), which we define as
\[
{\sf j}(\mathcal X)\coloneqq[\X,\sfd,(\mm+2\delta_p)/3]\quad\text{ for every }\mathcal X=[\X,\sfd,\mm,p]\in\mathbb X_{1,pt},
\]
is a topological embedding.
\end{proposition}
\begin{proof}
The implication \({\rm i)}\Longrightarrow{\rm ii)}\) readily follows from the observation that if
\(\iota_n(p_n)\to\iota_\infty(p_\infty)\), then \(\delta_{\iota_n(p_n)}\rightharpoonup\delta_{\iota_\infty(p_\infty)}\).
To prove the implication \({\rm ii)}\Longrightarrow{\rm i)}\), assume \([\X_n,\sfd_n,\mu_n]\to[\X_\infty,\sfd_\infty,\mu_\infty]\) in the
mG sense, where we set \(\mu_n\coloneqq(\mm_n+2\delta_{p_n})/3\in\mathscr P(\X_n)\) for all \(n\in\bar\N\). Given \(\varepsilon>0\), we have
\[\begin{split}
\frac{2}{3}&=\frac{2}{3}\delta_{\iota_\infty(p_\infty)}(B(\iota_\infty(p_\infty),\varepsilon))\leq
(\iota_\infty)_\#\mu_\infty(B(\iota_\infty(p_\infty),\varepsilon))\leq\liminf_{n\to\infty}(\iota_n)_\#\mu_n(B(\iota_\infty(p_\infty),\varepsilon))\\
&\leq\frac{1}{3}+\frac{2}{3}\liminf_{n\to\infty}\delta_{\iota_n(p_n)}(B(\iota_\infty(p_\infty),\varepsilon)),
\end{split}\]
which gives \(\{\iota_n(p_n):n\geq n_\varepsilon\}\subseteq B(\iota_\infty(p_\infty),\varepsilon)\) for some \(n_\varepsilon\in\N\).
Then \(\iota_n(p_n)\to\iota_\infty(p_\infty)\), thus also
\[
(\iota_n)_\#\mm_n=3(\iota_n)_\#\mu_n-2\delta_{\iota_n(p_n)}\rightharpoonup 3(\iota_\infty)_\#\mu_\infty-2\delta_{\iota_\infty(p_\infty)}
=(\iota_\infty)_\#\mm_\infty.
\]
All in all, we have proven that \([\X_n,\sfd_n,\mm_n,p_n]\to[\X_\infty,\sfd_\infty,\mm_\infty,p_\infty]\) in the pmG sense.
\end{proof}

Given any \(k\in\mathbb Z\), we define the Lipschitz cut-off function \(\eta_k\colon[0,+\infty)\to[0,1]\) as
\[
\eta_k(t)\coloneqq(2-2^{-k}t)^+\wedge 1\quad\text{ for every }t\in[0,+\infty).
\]
To any \(\mathcal X=[\X,\sfd,\mm,p]\in\mathbb X_{pt}\) we associate the spaces
\((\mathcal X^k)_{k\in\mathbb Z}\subseteq\mathbb X_{1,pt}\) given by
\[
\mathcal X^k\coloneqq[\X,\sfd,\mm^{p,k},p],\quad\text{ where we set }
\mm^{p,k}\coloneqq\bigg(\int\eta_k(\sfd(p,\cdot))\,\d\mm\bigg)^{-1}\eta_k(\sfd(p,\cdot))\mm\in\mathscr P(\X).
\]
In the next result, \(\mathbb X_1\) is equipped with the box topology and
\(\mathbb X_1^{\mathbb Z}\times\R^{\mathbb Z}\) with the product topology.
\begin{theorem}\label{thm:embed_for_pmG}
The map \(\Psi\colon(\mathbb X_{pt},\tau_{\rm pmG})\to\mathbb X_1^\mathbb Z\times\R^{\mathbb Z}\), which we define as
\[
\Psi(\mathcal X)\coloneqq\bigg(({\sf j}(\mathcal X^k))_{k\in\mathbb Z},
\bigg(\log\int\eta_k(\sfd(p,\cdot))\,\d\mm\bigg)_{k\in\mathbb Z}\bigg)
\quad\text{ for every }\mathcal X=[\X,\sfd,\mm,p]\in\mathbb X_{pt},
\]
is a topological embedding.
\end{theorem}
\begin{proof}
We recall from \cite[Theorem 3.15]{GigliMondinoSavare15} that the topology \(\tau_{\rm pmG}\) is induced by
the distance \({\rm p}\mathbb G_{\rm W}\) defined in \cite[Definition 3.13]{GigliMondinoSavare15}.
Given any \((\mathcal X_n)_{n\in\bar\N}\subseteq\mathbb X_{pt}\), by unpacking the definition of \({\rm p}\mathbb G_{\rm W}\)
one can easily check that \(\lim_n{\rm p}\mathbb G_{\rm W}(\mathcal X_n,\mathcal X_\infty)=0\) if and only if
\(\lim_n{\rm G}_W({\sf j}(\mathcal X_n^k),{\sf j}(\mathcal X_\infty^k))=0\) and
\(\lim_n\log\int\eta_k(\sfd_n(p_n,\cdot))\,\d\mm_n=\log\int\eta_k(\sfd_\infty(p_\infty,\cdot))\,\d\mm_\infty\) for every \(k\in\mathbb Z\).
In view of Remark \ref{rmk:GW-dmG} and Proposition \ref{prop:embd_Xpt_to_X}, we conclude that \(\Psi\) is a topological embedding.
\end{proof}
\begin{remark}\label{rmk:about_p-Wasserstein}{\rm
Throughout the paper, we have considered only truncated \(1\)-Wasserstein distances, solely for convenience of presentation.
However, all the results extend straightforwardly to truncated \({\rm p}\)-Wasserstein distances for any fixed exponent
\({\rm p}\in[1,\infty)\). More precisely, one may consider
\[
W_{\rm p}(\mu,\nu)\coloneqq\inf_{\pi\in{\rm Adm}(\mu;\nu)}\bigg(\int\sfd^{\rm p}\wedge 1\,\d\pi\bigg)^{1/{\rm p}}
\quad\text{ for every }\mu,\nu\in\mathscr P(\X);
\]
note that \(W_{\rm 1}=W\). All the qualitative results of this paper remain valid with this choice. The quantitative
results can easily be adapted as well: denoting by \(\sfd_{{\rm p},{\rm mG}}\), \({\rm D}_{\rm p}\) and \({\rm G}_{{\rm p},W}\)
the natural \({\rm p}\)-analogues of \(\sfd_{\rm mG}\), \({\rm D}\) and \({\rm G}_W\), respectively, one obtains that
\(\sfd_{{\rm p},{\rm mG}}={\rm D}_{\rm p}={\rm G}_{{\rm p},W}\), while the estimate for the box distance given in Theorem
\ref{thm:box_equiv_mG} becomes \(\frac{1}{2}\square^{({\rm p}+1)/{\rm p}}\leq\sfd_{\rm p,mG}\leq(\frac{3}{2}\square)^{1/{\rm p}}\).
\fr}\end{remark}
\begin{remark}{\rm
In \cite{Sturm2023}, Sturm investigated the geometry of the \emph{space of metric measure spaces}. We expect that the techniques
developed in the present work may also prove useful in this context. For instance, by suitably adapting our arguments, one can
plausibly show that the structures considered by Sturm are metric quotients of spaces of probability measures on 
\(\mathbb U\). Such identifications may, in turn, provide a useful tool for further understanding their intrinsic geometry.
\fr}\end{remark}
\def\cprime{$'$} \def\cprime{$'$}

\end{document}